\documentclass[1p,11pt]{elsarticle}

\usepackage[english]{babel}
\usepackage[utf8]{inputenc}
\usepackage[svgnames]{xcolor}
\usepackage[T1]{fontenc}
\usepackage{wrapfig} 

\usepackage{fancyhdr} 
\usepackage{lastpage} 
\usepackage{here}
\usepackage{color}
\usepackage{ragged2e}
\usepackage{array}
\usepackage{indentfirst} 

\usepackage{setspace}
\usepackage{float}
\usepackage[hidelinks]{hyperref} 
\usepackage{amssymb} 
\usepackage{amsmath,amsfonts} 
\usepackage{amsthm}
\usepackage{amstext}

\usepackage{stmaryrd} 
\usepackage{mathtools} 
\usepackage{dsfont} 
\usepackage{mathrsfs}
\usepackage{enumitem} 
\usepackage{subcaption} 
\usepackage{multicol} 
\usepackage{multirow} 

\newtheoremstyle{notation}
  {}
  {}
  {}
  {}
  {\bfseries}
  {}
  {7pt}
  {\thmname{#1}\thmnumber{ #2}.\textnormal{\thmnote{ (#3)}}}
  
\newtheoremstyle{exem}
  {}
  {}
  {}
  {}
  {\bfseries}
  {}
  { }
  {\textit{\thmname{#1}}\thmnumber{ #2}.\textnormal{\thmnote{ (#3)}}}

\theoremstyle{plain}
\newtheorem{theo}{Theorem}[section]
\newtheorem*{theo*}{Theorem}

\newtheorem{lemme}[theo]{Lemma}
\newtheorem*{lemme*}{Lemma}

\newtheorem*{sublemme*}{Sublemma}

\newtheorem*{prop*}{Proposition}

\newtheorem{defi}[theo]{Definition}
\newtheorem*{defi*}{Definition}

\theoremstyle{definition}

\theoremstyle{notation}

\theoremstyle{remark}

\numberwithin{equation}{section}
\numberwithin{figure}{section} 
\counterwithin{table}{section}

\newcommand{\seminorm}[2][]{\left| #2 \right|_{#1}} 

\newcommand{\N}{\mathbb{N}} 
\newcommand{\Z}{\mathbb{Z}} 
\newcommand{\R}{\mathbb{R}} 

\allowdisplaybreaks 

\newcommand{\DocumentTitle}{A non-vanishing property for tensor products of wavelets}
\newcommand{\DocumentTitleShort}{Non-vanishing tensor products of wavelets}
\author[1]{Quentin Rible}
\author[1]{St{\'e}phane Seuret}
\affiliation[1]{Univ Paris Est Creteil, Univ Gustave Eiffel, CNRS, LAMA UMR8050, F-94010 Creteil, France}
\date{}

\title{\DocumentTitle}

\begin{document}

\setlength{\abovedisplayskip}{5pt}
\setlength{\belowdisplayskip}{5pt}


\begin{abstract}
We prove that, given a wavelet $\psi$, it is possible to choose some multi-integers $(p_j=(p_{j,1},...,p_{j,d}))_{j \in \Z} \in \Z^d$ such that, for every $x=(x_1,...,x_d) \in \R^d$, for infinitely many integers $j$, the tensorized wavelet $\prod_{i=1}^d \psi(2^j x_i-p_{j,i})$ does not vanish at $x$. This non-vanishing property is essential for analyzing some generic regularity properties in certain Sobolev and Besov spaces. The proof relies on an assumption regarding the zeros of $\psi$, which we numerically verify for the first Daubechies wavelets.
\end{abstract}

\maketitle

\section{Introduction}
In this chapter, we study certain cancellation properties of wavelets and of their periodized versions, with a particular focus on Daubechies wavelets.

A fundamental question in wavelet analysis is the localization of the zeros of a wavelet \(\psi : \R \to \R\). Fix $(\phi,\psi)$ a couple of scaling function and its associated wavelet, so that the family of dilations and translations
\[
    \{\psi_{j,k}(\cdot) := 2^{j/2}\psi(2^j \cdot - k)\}_{j,k \in \Z}
\]
of $\psi$ forms an orthonormal basis of $L^2(\R)$.
For a function $f \in L^2(\R)$, its wavelet coefficients with respect to this basis are denoted by 
\begin{equation*}
    c_k= \int_\R f(x)\phi(x-k)dx \ \mbox{ and } \ d_{j,k} = \int_{\R} f(x)\,\psi_{j,k}(x)\,dx.
\end{equation*}
For a sufficiently regular function $f:\R\to\R$, written as  
\[
    f = \sum_{k\in \Z} c_k\phi(\cdot -k) + \sum_{j\geq 0} \sum_{j,k \in \Z} d_{j,k}\,\psi_{j,k}
\]
where the equality holds pointwise, the fact that all $\psi_{j,k}(x) = 0$ for every $j\geq 0$ and $k\in \Z$ at some point $x\in \R$ would imply that none of the wavelet coefficients $d_{j,k}$ contribute to the value of $f(x)$, which is a degenerate situation. 

This observation leads to the following natural question: if $\psi(x)=0$, does there necessarily exist an integer $k\in\Z$ such that $\psi(x-k)\neq 0$?
If this were not the case, as observed before, the point $x$ would exhibit a remarkable degeneracy, since all coefficients $d_{0,k}$ would be irrelevant to the value of $f(x)$. Such a phenomenon would be unexpected, and we conjecture that it does not occur for general wavelets.
We provide numerical evidence showing that this property fails for the first 45 Daubechies wavelets.

Obtaining a rigorous analytical proof of this observation remains an open problem, we have not seen such a proof in the literature. Results closed in this question were obtained in \cite{Reyes:2005:Zero}, where the distribution of zeros of finite sums of wavelets is investigated, highlighting how their structure constrains where these zeros can occur and \cite{Temme:1997:polynomial_daubechies, Karam-Mansour:2012:roots_daubechies, Karam:2012:zeros_daubechies, Karam:2010:zeros_daubechies}, where the location of zeros of polynomials associated with Daubechies orthogonal and biorthogonal wavelets is analyzed, showing they are constrained (notably within the unit disk) to ensure perfect reconstruction. 

One easily shows that under a very weak assumption (see below), for a fixed point $x\in\R$, it is impossible to have $\psi_{j,k}(x)=0$ for all $(j,k) \in \Z^2$.
Consequently, even in the extreme situation where $\psi(x-k)=0$ for all $k\in\Z$, there must exist at least one pair $(j_x,k_x)$ such that $\psi_{j_x,k_x}(x) \neq 0$.
It is therefore natural to ask whether the scale parameter $j_x$ can be chosen uniformly bounded with respect to $x\in\R$.

The main focus of this chapter is to investigate whether analogous phenomena arise in the setting of tensor-product wavelets with compact supports.
We prove that, under a mild and natural assumption, for every point $x = (x_1,\ldots,x_d) \in \R^d$, there exist infinitely many pairs $(j,k)\in\N \times \Z^d$ such that
\[
    \psi(2^j x_i - k_i) = 0 \quad \text{simultaneously for all } i = 1,\ldots,d.
\]
We make precise the notion of “infinitely many” in this context and show, moreover, that the translation parameters $k_i$ can be chosen independently of $x$.
\medskip

We introduce a property that, according to the previous considerations, is verified by many compactly supported wavelets.

\begin{defi}\label{defi:Psi_spectrum_saturation_func}
    A real function $\psi:\R\to\R$ satisfies the property $(R)$ if:
    \begin{enumerate}[label=$(R_{\arabic*})$]
        \item $\psi\in\mathcal{C}^1(\R)$, and there is an integer $K_{\psi}\in\N^*$ such that $\overline{\operatorname{supp}(\psi)} = [0,K_{\psi}]$.
        \item The cardinality of the set $\mathcal{Z}(0,0):=\psi^{-1}(\{0\})\bigcap[0,K_{\psi}]$ of zeros of $\psi$ is finite.
        {\item\label{item:defi:Psi_spectrum_saturation_func:func_S} for all $x\in[0,1]$, $\displaystyle S(x) := \sum_{k=0}^{K_{\psi}-1} \seminorm{\psi(x+k)} >0$.}
    \end{enumerate}
\end{defi}
 
This property is satisfied by at least the Daubechies' wavelet of order up to 45 \cite{Daubechies:1988:compactly_supported_wavelets,Daubechies:1992:ten_lectures}, see Section \ref{sec-numeric}. 
It is an interesting question to prove (not only numerically) that large classes of wavelets, including Daubechies', satisfy (R), we will study it in an upcoming work. Property $(R)$ implies the following fact, which is the main result of the chapter.

\begin{theo}\label{theo:lower_bound_periodised_wavelet_function}
    Let $\psi:\R\to\R$ be a function satisfying property $(R)$, and let
    \begin{equation}\label{equation:defi_G}
        G:x\in\R \mapsto \sum_{p\in\Z}\psi(x-pK_{\psi}).
    \end{equation}
    be the periodized version of $\psi$.
    
    There is $\alpha>0$ with the following property: for every $d \in \N^*$, there exists an integer $N(d)$ and a sequence of multi-integer $(\overline{p_j}=(p_{j,1},\ldots,p_{j,d}))_{j\in\N}$ such that for every $x\in\R^{d}$ and $J\in \N$. there is an integer $j\in \{J,J+1,..., J+N(d)\}$ such that
    \begin{equation}\label{eq-minor}
        \prod_{i=1}^{d} |G(2^{j} x_i - p_{j,i})| \geq \alpha^{d}.
    \end{equation}
\end{theo}

In words, Theorem~\ref{theo:lower_bound_periodised_wavelet_function} asserts that, for any dimension $d$, there exists a sequence of multi-indices $(\overline{p_j})_{j\in\N}$, depending on the function $\psi$ only, such that for every $x\in\R^d$, the translated and dilated versions of the periodized function
\begin{equation*}
    G^1_{j,p_{j,i}}:=G(2^j \cdot - p_{j,i})
\end{equation*}
do not vanish simultaneously at the coordinates $x_i$ for infinitely many values of $j$.
More precisely, this non-vanishing property occurs at least once every $N(d)+1$ scales $j$, hence very regularly, and not only the translates do not vanish at $x$, but we deduce from the proof that they all are uniformly bounded by below. The independence with respect to $x$ of the sequence $(\overline{p_j})_{j\in \N} $ is important and noticeable.

\medskip
 
Beyond the intrinsic interest of the question itself, Theorem \ref{theo:lower_bound_periodised_wavelet_function} plays a crucial role in the analysis of the regularity of traces of functions. More precisely, let $d,d' \in \N^*$ and $D=d+d'$. For every $a\in \R^{d'}$, call $H_a=\R^d\times\{a\}$ the horizontal affine subspace of $\R^D$. The trace of a continuous function $f:\R^D\to\R$ along $H_a$ is then $f_a: y\in \R^d \mapsto f_a(y):= f(y,a)$, which is also continuous. It is proved in \cite{Rible-Seuret:2026:prevalence_traces, Aubry-Maman-Seuret:2013:Traces_besov_results} that in a homogeneous Besov space $B^{s}_{p,q}(\R^d)$ or in an inhomogeneous Besov space $B^\xi_{p,q}(\R^d)$, there is a prevalent set of functions whose multifractal properties can be explicitly determined for most of the parameters $a\in \R^{d'}$. These results are based on the fact that given a specific function $f:\R^D\to\R$ with wavelet coefficients $(d_{j,{\bf k}})_{j\in\N, {\bf k}\in \Z^D}$, the wavelet coefficients $(d_{j,{\bf k'}}(a))_{j\in \Z, {\bf k'}\in \Z^d}$ of its trace $f_a$ depend explicitly on $G$. Without entering into too many details, when for instance $\psi$ is a compactly supported wavelet with $\overline{\operatorname{supp}(\psi)}=[0,N]$ and when 
\begin{equation*}
    d_{j,k}=
    \begin{cases} 
        2^{-j\alpha} & \mbox{ if }{\bf k}=(k_1..., k_D) \mbox{ is such that } 
        \begin{cases} 
            \forall i\in \{1,...,d\}, k_i=p_{j,i} \mod N \\
            \forall i\in \{d+1,...,D\}, k_i= 0
        \end{cases}\\
        0 & \mbox{ otherwise,}
    \end{cases}
\end{equation*}
then one has 
\begin{equation*}
    d_{j,k}(a) = 2^{-j\alpha}\prod_{i=1}^{d} |G(2^{j} a_i - p_{j,i}) .
\end{equation*}
Then, in order to estimate the regularity properties of $f_a$, it is mandatory to estimate $d_{j,k}(a)$, and the non-vanishing property \eqref{eq-minor} is key. 

\medskip

The rest of this note is devoted to the proof of Theorem \ref{theo:lower_bound_periodised_wavelet_function} in Section \ref{section:property-wavelet}, and to the numerical verification of the property (R) in Section \ref{sec-numeric} for Daubechies wavelets. Observe that Theorem \ref{theo:lower_bound_periodised_wavelet_function} does not require $\psi$ to be a wavelet, however for our application purposes it is needed to verify it for compactly supported wavelets.

\section{Proof of Theorem \ref{theo:lower_bound_periodised_wavelet_function}}\label{section:property-wavelet}

We assume that the function $\psi$ satisfies Property $(R)$ given in Definition~\ref{defi:Psi_spectrum_saturation_func}.

Since $\psi$ is compactly supported on $[0, K_{\psi}]$, the function $G$ (recall \eqref{equation:defi_G}) is $K_{\psi}$-periodic, with $G(0) = G(K_{\psi}) = 0$ and $G \in \mathcal{C}^1(\R)$.

\begin{defi}\label{defi:Gd}
    For every $\overline{p_j} = (p_{j,1}, \ldots, p_{j,d}) \in \Z^d$, the function $G_{j,\overline{p_j}}^{d}:\R^d \to \R$ is defined by
    \begin{equation}\label{equation:Gd}
        \text{for every } x = (x_1, \ldots, x_d) \in \R^d, \quad G_{j,\overline{p_j}}^{d}(x) = \prod_{i=1}^{d} G(2^j x_i - p_{j,i}).
    \end{equation}
\end{defi}

\noindent
Note that with this notation, $G_{j,\overline{p_j}}^{d}$ can be rewritten as $\prod_{i=1}^{d} G^1_{j,p_{j,i}}$.

Our goal is to prove Theorem~\ref{theo:lower_bound_periodised_wavelet_function}, which asserts that the integers $\overline{p_j}$ can be chosen such that, for every $x$, at least one of the functions $(G_{J+\ell,\overline{p_{J+\ell}}}^{d})_{\ell \in \{0,\ldots,N(d)\}}$ is non-zero and bounded below (in absolute value) by a constant depending on $\alpha$ at $x$. Moreover, this occurs infinitely often.

\medskip

We begin with some preliminary remarks and definitions.

Since $\psi \in \mathcal{C}^1(\R)$, $\psi'$ is continuous and bounded on $[0, K_{\psi}]$. Consequently, $G'$ is bounded on $\R$ as a periodized version of $\psi$. We denote
\[
    M_G := \max_{x \in [0, K_{\psi}]} |G'(x)|.
\]

For the remainder of this section, we fix $d \in \N^*$ and consider an increasing family of integers $N(1), \ldots, N(d)$, to be specified later.

\begin{defi}
    For integers $j$ and $p$, let $\mathcal{Z}(j,p)$ be the set of zeros of $G(2^j x - p)$ on $[0, K_{\psi}]$. For $\omega > 0$, the $\omega$-neighborhood of $\mathcal{Z}(j,p)$ is defined as
    \[
        \mathcal{N}_{\omega}\mathcal{Z}(j,p) = \bigcup_{z \in \mathcal{Z}(j,p)} (z - \omega, z + \omega).
    \]
\end{defi}

We start with a basic lemma, which shows that for any $x \in [0, K_{\psi}]$, at each scale $j$, there exists an integer $p$ such that $|G_{j,p}(x)|$ is uniformly bounded below.

\begin{lemme}\label{lemme:lower_bound_one_dim_periodised_wavelet_function}
    There exists $\eta > 0$ such that for every $x \in \R$ and every $j$, there exists $p_{j,x} \in \{0, \ldots, K_{\psi} - 1\}$ satisfying
    \begin{equation}\label{eq-minG}
        |G^1_{j,p_x}(x)| = |G(2^j x - p_x)| \geq \eta.
    \end{equation}
\end{lemme}

\begin{proof}
    The function $S$, defined in~\ref{item:defi:Psi_spectrum_saturation_func:func_S}, is continuous on $[0,1]$, hence bounded and attains its bounds. Let $\widetilde{\eta} = \min_{x \in [0,1]} S(x) > 0$.

    By construction, for $x \in [0, K_{\psi}]$, we have $G(x) = \psi(x)$. Thus, for $x \in [0,1]$, $S(x) = \sum_{p=0}^{K_{\psi}-1} |\psi(x+p)| = \sum_{p=0}^{K_{\psi}-1} |G(x+p)| \geq \widetilde{\eta}$. Therefore, for all $x \in [0, K_{\psi}]$, there exists $p_x \in \{0, \ldots, K_{\psi}-1\}$ such that
    \begin{equation}\label{eq-min}
        |G(x - p_x)| \geq \frac{\widetilde{\eta}}{K_{\psi}} =: \eta.
    \end{equation}
    The $K_{\psi}$-periodicity of $G$ yields~\eqref{eq-minG} for $j = 0$.

    For $j \geq 1$, we apply the above result to $x' = 2^j x$ to obtain~\eqref{eq-minG}.
\end{proof}

We define the following sequence of integers: let $N = \#\mathcal{Z}(0,0)$ and set
\begin{equation}\label{def-N(d)}
    N(1) = 2(N + 1) \quad \text{and for } d' \geq 2, \quad N(d') = 2(N(d'-1) + 1)^{d'} 2^{N(d'-1)}.
\end{equation}

Recall the parameter $\eta$ from Lemma~\ref{lemme:lower_bound_one_dim_periodised_wavelet_function}, and define
\begin{equation}\label{choice-omega}
    \varepsilon_d := \frac{\eta}{2^{N(d) + 1} M_G}, \quad \varepsilon'_d := \min_{x,y \in \mathcal{Z}(0,0)} \frac{|x - y|}{4 \times 2^{N(d)}}, \quad \text{and} \quad \omega := \min\{\varepsilon_d, \varepsilon'_d\}.
\end{equation}

Since $\mathcal{Z}(0,0)$ is a finite set of cardinality $N$, we denote its elements by
\[
z_1 = 0 < z_2 < \cdots < z_{N-1} < z_N = K_{\psi}.
\]

\begin{defi}
    For every integers $j$ and $p$, let $\mathcal{Z}(j,p)$ be the set of zeros of $G(2^j x - p)$ on $[0, K_{\psi}]$. The $\omega$-neighborhood of $\mathcal{Z}(j,p)$ is
    \[
        \mathcal{N}_{\omega}\mathcal{Z}(j,p) = \bigcup_{z \in \mathcal{Z}(j,p)}B(z,\omega) , \ \ \mbox{ where } B(z,\omega)= (z - \omega, z + \omega).
    \]
\end{defi}

\noindent
We make the following observations:
\begin{itemize}
    \item For every $j$, the elements of $\mathcal{Z}(j,p)$ are precisely those real numbers that can be written as $2^{-j}(k K_{\psi} + p + z_i)$, where $z_i \in \mathcal{Z}(0,0)$ and $k \in \Z$.
    \item For every $j \in \{0, \ldots, N(d)\}$, two elements of $\mathcal{Z}(j,p)$ are separated by at least $4\varepsilon'_d$.
    \item By our choice of $\omega$, for every $j \in \{0, \ldots, N(d)\}$, each set $\mathcal{N}_{\omega}\mathcal{Z}(j,p)$ is a union of disjoint open intervals, as any two elements of $\mathcal{Z}(j,p)$ are separated by at least $4\omega$.
\end{itemize}

We first consider the one-dimensional case.

\begin{lemme}\label{lem-dim1}
    There exists $\alpha > 0$ such that the following holds: for every integers $0 \leq J \leq N(d) - N(1)$ and $p_J \in \Z$, there exists a sequence $(p_j)_{j \geq J}$ such that for every $x \in \R$ and every $j \in \{J, \ldots, N(d) - N(1)\}$, there is an integer $\ell \in \{0, \ldots, N(1)\}$ with $|G^1_{j+\ell,p_{j+\ell}}(x)| \geq \alpha$.
\end{lemme}
\begin{proof}
    We work on the interval $I = [2^{-J} p_J, 2^{-J}(p_J + K_{\psi})]$. Since all functions $(G^1_{j,p})_{j \geq J, p \in \Z}$ are $2^{-J} K_{\psi}$-periodic, the result extends directly to $[0,1]$.

    By the above remark, $\mathcal{Z}(J,p_J) \cap I = \{2^{-J}(z_1 + p_J), \ldots, 2^{-J}(z_N + p_J)\}$ and this set has cardinality $N$. For simplicity, we introduce the notation:
    \begin{eqnarray}\label{eq-B0}
        \mbox{ for $m \in \{1,...,N\}$, } \ S_{m} & := & B(2^{-J}(z_{m}+p_{J}),\omega) \\
        \mbox{and } \ \ S_0& := & I \setminus \mathcal{N}_{\omega}\mathcal{Z}(J,p_J),
    \end{eqnarray}
    so that for every $x \in I$, there exists a unique index $m \in \{0, \ldots, N\}$ such that $x \in B(2^{-J}(z_m + p_J), \omega)$.

    \smallskip
    \textbf{(i)} If $x \in S_0$, then this set is a finite union of closed intervals on which $|G^1_{J,p_J}|$ is non-zero. Thus, $|G^1_{J,p_J}(x)| \geq \min\{|G^1_{J,p_J}(x)| : x \in B(2^{-J}(z_0 + p_J), \omega)\} =: \widetilde{\alpha} > 0$. Importantly, by auto-similarity and periodicity, $\widetilde{\alpha}$ does not depend on $J$ nor on $p_J$.

    \smallskip
    \textbf{(ii)} If $x \in S_m$ for some $m \in \{1, \ldots, N\}$, then by Lemma~\ref{lemme:lower_bound_one_dim_periodised_wavelet_function}, there exists $p_{J+m}$ such that~\eqref{eq-minG} holds for $G^1_{J+m,p_{J+m}}$ at the point $2^{-J}(z_m + p_J)$, i.e., $|G^1_{J+m,p_{J+m}}(2^{-J}(z_m + p_J))| \geq \eta$. By the mean value theorem,
    \[
        |G^1_{J+m,p_{J+m}}(x)| \geq \eta - |x - 2^{-J}(z_m + p_J)| 2^{J + m} M_G \geq \frac{\eta}{2},
    \]
    due to our choice of $\omega$ in~\eqref{choice-omega} and the restriction $J + m \leq N(d)$.

    This allows us to fix $p_j$ for all $j \in \{J, \ldots, J + N\}$, ensuring the conclusion of Lemma~\ref{lem-dim1} with $\alpha = \min(\eta/2, \widetilde{\alpha})$ for these integers.

    \medskip

    Assume now that $p_j$ is constructed for all $j \in \{J, \ldots, J + k(N+1) - 1\}$ for some $k \geq 1$. We then set $p_{J+k(N+1)} = 0$ and apply the above argument to $G^1_{J+k(N+1),0}$ to obtain the values of $p_j$ for $j \in \{J + k(N+1), \ldots, J + (k+1)(N+1) - 1\}$.

    This process can be repeated until $J + (k+1)(N+1) - 1 \leq N(d)$, providing the entire sequence $(p_j)_{J \leq j \leq N(d)}$. Finally, with $N(1) = 2(N + 1)$ as in~\eqref{def-N(d)}, for every $x \in [0, K_{\psi}]$ and every $J \leq j \leq N(d) - N(1)$, the sequence of integers $\{j, j+1, \ldots, j+N(1)\}$ contains a subset of the form $j \in \{J + k(N+1), \ldots, J + (k+1)(N+1) - 1\}$ on which it is guaranteed that one of the $j$'s satisfies $|G^1_{j,p_j}(x)| \geq \alpha$.
\end{proof}

We now outline the approach for the two-dimensional case, which serves as the base case for an induction on the dimension. The key idea is to ensure that one coordinate satisfies $|G^1_{j,p_1}(x_i)| > \alpha$ and to apply Lemma~\ref{lem-dim1} to the other coordinate.

\begin{lemme}\label{lem-dim2}
    For the same $\alpha$ as in Lemma \ref{lem-dim1}, the following property holds: for every integer $0\leq J\leq N(d)-N(2)$
    and $\overline{p_J}=(p_{J,1},p_{J,2}) \in \Z^2$, there is a sequence $(\overline{p_j}=(p_{j,1},p_{j,2}))_{j\geq J}$ such that for every $x\in \R^2$, for every $ j\in\{ J,..., N(d)-N(2)\}$, there is an integer $\ell \in \{0,..., N(2) \}$ with $|G^2_{j+\ell,\overline{p_{j+\ell}}}(x)|\geq \alpha^2$.
\end{lemme}
\begin{proof}
    We work on $I_2=[ 2^{-J} p_{J,1},2^{-J} (p_{J,1}+K_\psi)]\times [ 2^{-J} p_{J,2},2^{-J} (p_{J,2}+K_\psi)]$ and the results are extended on $\R^2$ by periodicity. Fix $\overline{p_J}=(p_{J,1},p_{J,2})$. We partition $I_2$ into the following $(N+1)^2$ disjoint subsets defined as follows: for $(m_1,m_2)\in \{0,1,2,...,N\}^2$,
    \begin{equation*}
        S_{m_1,m_2} = S_{m_1}\times S_{m_2},
    \end{equation*}
    where the notation \eqref{eq-B0} is used again.
    
    \smallskip
    \textbf{$\bullet$ Case $(m_1,m_2)=(0,0)$:} we apply the argument (i) of the proof of Lemma \ref{lem-dim1} to each coordinate to get that if $x=(x_1,x_2)\in S_{0,0}$, $|G^2_{J,(p_{J,1},p_{J,2})}(x_1,x_2)|= |G_{J,p_{J,1}}(x_1)G_{J,p_{J,2}}(x_2)|\geq \alpha^2$.
    
    \smallskip
    \textbf{$\bullet$ Case $m_1\neq 0$ and $m_2 \neq 0$:} we call $\ell=m_1+(m_2-1)N$. We apply Lemma \ref{lem-dim1} to each coordinate $i\in \{1,2\}$ with $G^1_{J,p_{J,i}}$ and $j=J+\ell$ to obtain two integers $p_{J+\ell, i}$, $i=1,2$, such that if $x=(x_1,x_2)\in S_{m_1,m_2}$, then $|G^2_{J+\ell ,(p_{J+\ell,1},p_{J+\ell,2})}(x_1,x_2)|= |G^1_{J,p_{J+\ell,1}}(x_1)G^1_{J,p_{J+\ell,2}}(x_2)|\geq \alpha^2$.
    
    This requires $N^2$ generations $j$ to treat all situations in this case, and we set $\overline{p_{J+\ell}}=(p_{J+\ell,1},p_{J+\ell,2})$ for $\ell\in \{1,...,N^2\}$.
    
    \smallskip
    \textbf{$\bullet$ Case $m_1\neq 0$ and $m_2 = 0$:} We start with $m_1=1$. Up to now, the $\overline{p_j}$ are found for $J\leq j\leq J+N^2$.
    
    For $1\leq \ell\leq N(1)$, choose $p_{J+N^2+\ell,1}$ such that \eqref{eq-minG} holds true for $j=J+N^2+\ell$ and $z_{m_1}$. The same argument as in item (ii) in the proof of Lemma \ref{lem-dim1} gives that for $x_1\in B(2^{-J}(z_{1}+p_J),\omega)$, $|G^1_{j+N^2+\ell,p_{j+N^2+\ell,1 }}(x_1)|\geq \eta/2$ for all $j\in \{j=J+N^2+1,j+N^2+N(1)\}$. This treats the first coordinate, i.e. we ensure that $G^1_{j+N^2+\ell,p_{j+N^2+\ell,1 }}(x_1)$ is not zero and uniformly bounded by below.
    
    To deal with the second coordinate, we apply Lemma \ref{lem-dim1} to $G^1_{J,p_{J,2}}$ with $j=J+N^2+1$: for every $x_2\in [0,K_\psi]$, there exist integers $(p_{j+\ell,2})_{\ell\in \{0,...,N(1)\}}$ such that $|G^1_{j+\ell,p_{j+\ell,2}}(x_2)|\geq \alpha$ for some $\ell\in \{0,...,N(1)\}$.
    
    So, setting $\overline{p_j}=(p_{j,1},p_{j,2})$ with the above values and $j\in \{J+N^2+1,...,J+N^2+N(1)\}$, we ensured that for every $x\in S_{1,m_2}$, for at least one of the integers $j\in \{J+N^2+1,...,J+N^2+N(1)\}$ one has $|G^2_{j,\overline{p_j}}(x)|\geq \alpha^2$.
    
    \smallskip
    
    We apply the same method $N$ times to treat all situations $S_{m_1,0}$, $m_1\in\{1,...,N\}$. This requires $ N\cdot N(1)$ generations.

    \smallskip
    {\bf $\bullet$ Case $m_1= 0$ and $m_2 \neq 0$:} This situation is symmetric to the first one, and requires again to fix the values of $\overline{p_j}$ on $ N\cdot N(1)$ generations.
    
    \medskip
    
    Gathering the above, one needs the generations $J\leq j \leq J+N^2+2N\cdot N(1)$ to ensure that for every $x\in [0,K_\psi2^{-J}]^2$, $|G^2_{j,\overline{p_{j}}}(x)|\geq \alpha^2$ for (at least) one of them.
    
    \medskip
    
    The argument then runs similarly as in Lemma \ref{lem-dim1}: assume that $\overline{p_j}$ is built for all $j\in \{J,...,J+k(N^2+4N\cdot N(1))-1\}$, for some $k\geq 0$.
    
    We then fix $\overline{p_{J+k(N^2+4N\cdot N(1))}}=(0,...,0)$, and apply to $G_{J+k(N^2+4N\cdot N(1)), (0,...,0)}$ the above argument to get the values of $\overline{p_j}$ for $j\in\{ J+k(N^2+4N\cdot N(1)),..., J+(k+1)(N^2+4N\cdot N(1))-1\}$. 
    
    We complete the sequence $\overline p_j$ until $j$ reaches $N(d)$. This ensures that $G^2_{j,\overline{p_j}}(x)$ is regularly non-zero (at most once every $2(2N^2+4N\cdot N(1))$ integers $j$).
     
    Recalling that $N(2)$ was fixed in \eqref{def-N(d)}, observe that 
    \begin{equation*}
        N(2) = 2(N(1)+1)^22^{N(1)}\geq 2(2N^2+4N\cdot N(1)),
    \end{equation*} 
    Lemma \ref{lem-dim2} is proved.
\end{proof}

We are now ready for the full induction.

\begin{lemme}\label{lem-dimd}
    For the same $\alpha$ as in Lemma~\ref{lem-dim1}, the following property holds for every $d' \in \{1, 2, \ldots, d\}$: for every integer $0 \leq J \leq N(d) - N(d')$ and $\overline{p_J} = (p_{J,1}, \ldots, p_{J,d'})\in \Z^{d'}$, there exists a sequence $(\overline{p_j} = (p_{j,1}, \ldots, p_{j,d'}))_{j \geq J}$ such that for every $x \in \R^{d'}$ and every $j \in \{J, \ldots, N(d) - N(d')\}$, there is an integer $\ell \in \{0, \ldots, N(d')\}$ with $|G^{d'}_{j+\ell, \overline{p_{j+\ell}}}(x)| \geq \alpha^{d'}$.
\end{lemme}

Observe in particular that for $d' = d$, Lemma~\ref{lem-dimd} applies only to $J = j = 0$. The arguments are essentially those of Lemma~\ref{lem-dim1}, extended to higher dimensions.

\begin{proof}
    Assume that the conclusion holds for all $1 \leq d'' < d'$.
    
    We work on $I_{d'} = \prod_{i=1}^{d'} [2^{-J} p_{J,i}, 2^{-J} (p_{J,i} + K_{\psi})]$, and the results are extended by periodicity. Fix $\overline{p_J} = (p_{J,1}, \ldots, p_{J,d'})$. The set $I_{d'}$ is partitioned into the $(N+1)^{d'}$ disjoint subsets
    \[
        S_{m_1, \ldots, m_{d'}} = S_{m_1} \times \cdots \times S_{m_{d'}},
    \]
    where $(m_1, \ldots, m_{d'}) \in \{0, 1, 2, \ldots, N\}^{d'}$.
    
    \smallskip
    \textbf{$\bullet$ Case $(m_1, \ldots, m_{d'}) = (0, \ldots, 0)$:}
    We apply item~(i) of the proof of Lemma~\ref{lem-dim1} to each coordinate. If $x = (x_1, \ldots, x_{d'}) \in S_{0, \ldots, 0}$, then
    \[
        |G^{d'}_{J, (p_{J,1}, \ldots, p_{J,d'})}(x_1, \ldots, x_{d'})| = |G^1_{J,p_{J,1}}(x_1) \cdots G^1_{J,p_{J,d'}}(x_{d'})| \geq \alpha^{d'}.
    \]
    
    \smallskip
    \textbf{$\bullet$ Case $m_i \neq 0$ for all $i \in \{1, \ldots, d'\}$:}
    As in Lemma~\ref{lem-dim2}, for each $d'$-uplet $(m_1, \ldots, m_{d'})$, we apply Lemma~\ref{lem-dim1} to each coordinate $i \in \{1, \ldots, d'\}$ with $G^1_{J,p_{J,i}}$ and $j = J + \ell$, where $\ell = \sum_{i=1}^{d'} (m_i - 1)N^{i-1} - 1$, to obtain $d'$ integers $p_{J+\ell, i}$, $i = 1, \ldots, d'$, such that if $x = (x_1, \ldots, x_{d'}) \in S_{m_1, \ldots, m_{d'}}$, then
    \[
        |G^{d'}_{J+\ell, (p_{J,1}, \ldots, p_{J,d'})}(x_1, \ldots, x_{d'})| = |G^1_{J,p_{J,1}}(x_1) \cdots G^1_{J,p_{J,d'}}(x_{d'})| \geq \alpha^{d'}.
    \]
    This requires $N^{d'}$ generations $j$.
    
    \smallskip
    \textbf{$\bullet$ There exists at least one non-zero index $m_i$ and $(m_1, \ldots, m_{d'}) \neq (0, \ldots, 0)$:}
    There are $(N+1)^{d'} - N^{d'} - 1$ such cases. Assume that all values of $\overline{p_j}$ have been constructed up to generation $\widetilde{J}$. We explain how to handle one of these cases $(m_1, \ldots, m_{d'})$.
    
    The idea is first to fix the integers $p_{j,i}$ for those indices $i$ such that $m_i \neq 0$ using Lemma~\ref{lem-dim1}, and then to treat the other indices via the induction hypothesis.
    
    Let $i_1, \ldots, i_k$ be the indices $i$ for which $m_i \neq 0$. For $1 \leq \ell \leq 2N(d' - k)$ and all $n \in \{1, \ldots, k\}$, choose $p_{\widetilde{J} + \ell, i_n}$ such that~\eqref{eq-minG} holds for $j = \widetilde{J} + \ell$, $G^1_{\widetilde{J} + \ell, p_{\widetilde{J} + \ell, i_n}}$, and $z_{m_{i_n}}$. As in Lemma~\ref{lem-dim2}, for each $x_{i_n} \in B(2^{-J}(z_{i_n} + p_{J,i_n}), \omega)$, we have $|G^1_{\widetilde{J} + \ell, p_{\widetilde{J} + \ell, i_n}}(x_{i_n})| \geq \eta/2$. Thus, for all such integers $j \in \{\widetilde{J} + 1, \ldots, \widetilde{J} + 2N(d' - k)\}$,
    \[
        \prod_{i \in \{i_1, \ldots, i_k\}} |G^1_{j,p_{j,i}}(x_i)| \geq \left(\frac{\eta}{2}\right)^k.
    \]
    
    To handle the other coordinates $i'_1, \ldots, i'_{d' - k}$, we apply the induction hypothesis to $\prod_{i \in \{i'_1, \ldots, i'_{d' - k}\}} G^1_{\widetilde{J}, p_{\widetilde{J}, i}}$ with $j = \widetilde{J} + 1$: there exists a finite sequence of multi-integers $(p_{\widetilde{J} + \ell, i'_1}, \break \ldots, p_{\widetilde{J} + \ell, i'_{d' - k}})_{\ell \in \{0, \ldots, N(d' - k)\}}$ such that for every
    \[
        x' = (x'_{i'_1}, \ldots, x'_{i'_{d' - k}}) \in \widetilde{I} := \prod_{i \in \{i'_1, \ldots, i'_{d' - k}\}} [2^{-J} p_{J,i}, 2^{-J} (p_{J,i} + K_{\psi})],
    \]
    we have
    \[
        \prod_{i \in \{i'_1, \ldots, i'_{d' - k}\}} |G^1_{\widetilde{J} + \ell, p_{j,i}}(x'_i)| \geq \alpha^{d' - k},
    \]
    for some $\ell \in \{0, \ldots, N(d' - k)\}$.
    
    Thus, setting $\overline{p_j} = (p_{j,1}, \ldots, p_{j,d'})$ with the above values and $j \in \{\widetilde{J} + 1, \ldots, \widetilde{J} + N(d' - k)\}$, we ensure that for every $x \in S_{m_1, \ldots, m_{d'}}$, one of the functions satisfies $|G^{d'}_{j, \overline{p_j}}(x)| \geq \left(\frac{\eta}{2}\right)^k \alpha^{d' - k} \geq \alpha^{d'}$ for at least one $j \in \{\widetilde{J} + 1, \ldots, \widetilde{J} + N(d' - k)\}$.
    
    \smallskip
    We apply the same method $(N+1)^{d'} - N^{d'} - 1$ times to address all situations, each requiring at most $N(d' - 1)$ values $\overline{p_j}$ to be fixed.
    
    At this step, we have fixed a number, say $\widetilde{N}$, of values of $\overline{p_j}$. We note that~\eqref{def-N(d)} ensures that
    \[
        N(d') = 2N(d' - 1)^{d'} 2^{N(d' - 1)} N^{d'} \geq 2(N^{d'} + N(d' - 1)(N + 1)^{d'}) \geq 2 \widetilde{N}.
    \]
    Finally, 
    we complete the sequence $\overline p_j$ similarly until $j$ reaches $N(d) $. 
     
    \medskip
    We complete the sequence $\overline{p_j}$ as follows: assume that for all $j \in \{J, \ldots, J + kN(d')\}$ for some $k \geq 0$. The above construction provides a method to fix $\overline{p_j}$ over the next $N(d')$ integers, i.e., for $j \in \{J + kN(d')+1,..., J + (k+1)N(d')\}$.
    
    The construction ensures that Lemma~\ref{lem-dimd} holds for dimension $d'$, completing the induction.
\end{proof}

We are now in a position to complete the proof of Theorem~\ref{theo:lower_bound_periodised_wavelet_function} and to construct the sequence $\overline{p_j}$.

\medskip
\textbf{$\bullet$ The case $J = 0$:}
This is precisely Lemma~\ref{lem-dimd} in dimension $d$ applied to $G^d_{(0, \ldots, 0)}$. This gives the values of the multi-integers $\overline{p_j}$ for $j = 1, \ldots, N(d)$.

\medskip
\textbf{$\bullet$ The case $J = kN(d)$ with $k \geq 1$:}
For every $j \in \{J + 1, \ldots, J + N(d)\}$, set $\overline{p_j} = \overline{p_{j - J}}$.

Let $x = (x_1, \ldots, x_d) \in [0, 2^{-J} K_{\psi}]^d$. We apply Lemma~\ref{lem-dimd} to $G^d_{0, \overline{p_J}}$ at the point $x' = (2^J x_1, \ldots, 2^J x_d)$. There exists $j_{x'} \in \{0, \ldots, N(d)\}$ such that $|G^d_{j_{x'}, \overline{p_{j_{x'}}}}(x')| \geq \alpha^d$. But $G^d_{j_{x'}, \overline{p_{j_{x'}}}}(x') = G^d_{j + j_{x'}, \overline{p_{j_{x'}}}}(x)$, hence the result for $x = (x_1, \ldots, x_d) \in [0, 2^{-j} K_{\psi}]^d$.

The $2^{-J}K_{\psi}$-periodicity of the functions $G^d_{J + \ell}$ allows us to extend the results to all $x \in \R^d$.

This concludes the proof of Theorem~\ref{theo:lower_bound_periodised_wavelet_function}.

\section{Numerical proof of (R)}\label{sec-numeric}

We plot the graph of the function $S$ (see Definition~\ref{defi:Psi_spectrum_saturation_func}) for several Daubechies wavelets, specifically, the graph of $S$ associated with $\psi = \text{db3}$ and $\psi = \text{db7}$ are given in Figure \ref{fig-db1}, and that of $\psi = \text{db45}$ is shown in Figure \ref{fig-db2}. The plots were generated using the Wavelet Toolbox of MATLAB, and they clearly illustrate the non-vanishing property of $S$, provided that the errors in estimating $\psi$ are controlled. Let us now discuss a few observations regarding these results.

\begin{figure}
    \begin{subfigure}[h]{.45\linewidth}
        \centering
        \includegraphics[width=\linewidth]{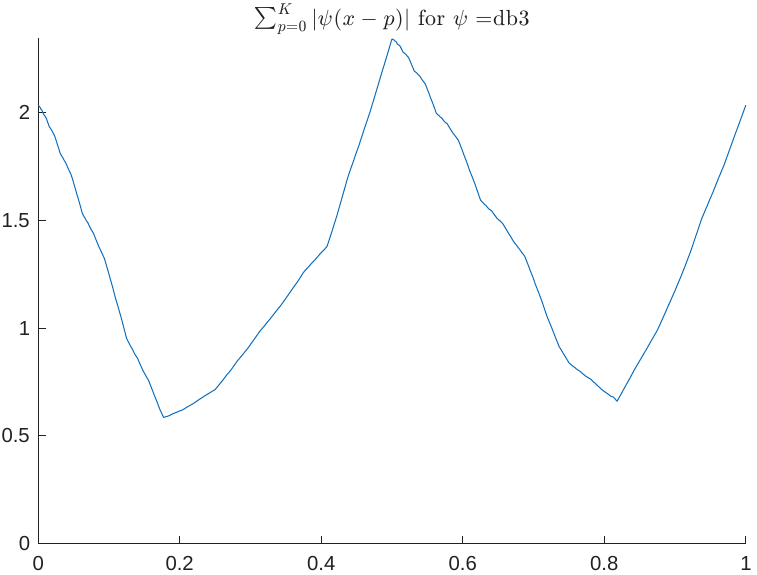}
        \caption{'db3' $\in C^1(\R)$ with 3 vanishing moments}
    \end{subfigure}
    \hfill
    \begin{subfigure}[h]{.45\linewidth}
        \centering
        \includegraphics[width=\linewidth]{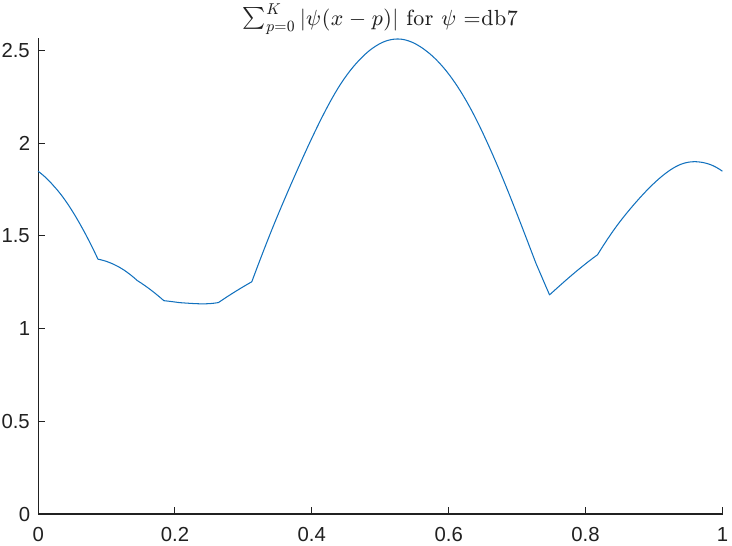}
        \caption{'db7' $\in C^2(\R)$ with 7 vanishing moments}
    \end{subfigure} 
    \caption{Daubechies wavelets 'db3' and 'db7' verify $(R)$. }
    \label{fig-db1}
\end{figure}

Let $\psi$ denote the Daubechies wavelet of order $p$, and $\varphi$ its associated scaling function. The relation between $\psi$ and $\varphi$ is
\begin{eqnarray*}
    \varphi(x) & =& \sqrt{2}\sum_{k\in \Z} h_k\, \varphi(2x-k)\\
    \psi(x) & = & \sqrt{2}\sum_{k\in \Z} g_k\, \varphi(2x-k),
\end{eqnarray*}
where $\{h_k\}_{k\in \Z} $ and $\{g_k\}_{k\in \Z} $ are the low-pass and high-pass filters.
The scaling function is computed via the cascade iteration
\[
    \varphi_{n+1}=T\varphi_n, \qquad (Tf)(x)=\sqrt{2}\sum_k h_k f(2x-k).
\]
An approximation of the wavelet, used by the standard algorithms, is given by
\[
    \psi_n(x) = \sqrt{2}\sum_k g_k\, \varphi_n(2x-k).
\]

The errors made when replacing $\psi$ by $\psi_n$ are of two types: the truncation errors and the roundoff errors.

Concerning the truncation error, it is known \cite{Daubechies:1992:ten_lectures, Cohen:2003:numerical_analysis_wavelet} that if $\varphi\in C^s(\R)$, then 
\[
    \|\varphi-\varphi_n\|_{C^r} \le C 2^{-ns}, \ \mbox { and so } \ \|\psi-\psi_n\|_{C^r} \le C 2^{-ns}.
\]
The constant $C$ is determined by the initial regularity of $ \varphi_0 $ used in the cascade algorithm and remains bounded by a small, controlled value.

Furthermore, floating-point perturbations inherent in computational imple\-mentations, affecting both the estimation of filter coefficients and all subsequent calculations, are on the order of $10^{-16}$. Even after 15 iterations, standard stability analysis of stationary subdivision schemes demonstrates that the cumulative error remains at most on the order of $10^{-14}$.

Collectively, provided that the wavelet belongs to $C^2$, the error after $n=15$ iterations of the cascade algorithm to estimate $\varphi$ and $\psi$ is less than $10^{-9}$. This result ensures the confidence provided by Figures~\ref{fig-db1} and~\ref{fig-db2} regarding the non-vanishing property of $S$ for Daubechies wavelets of reasonable order.

\begin{figure}
    \centering
    \includegraphics[scale=0.5]{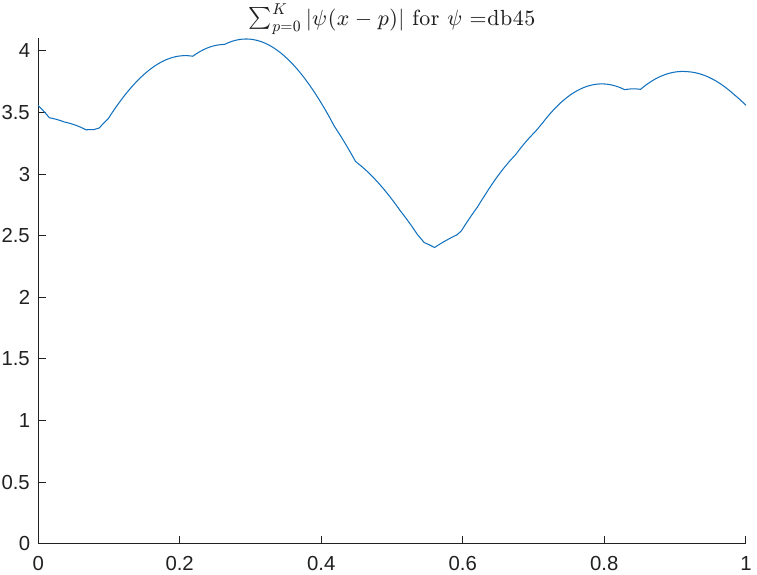}
    \caption{Function $S$ associated with 'db45': it is strictly positive.}
    \label{fig-db2}
\end{figure}

It is also worth noting that as the number $p$ of vanishing moments of the Daubechies wavelet $ \psi $ increases, the trade-off between the global regularity of $ \psi $ and the computational cost due to the larger filter banks becomes more challenging. Consequently, increasing the number of vanishing moments degrades the conditioning and reduces the number of reliable digits. Nevertheless, even for relatively high orders, the cascade algorithm in double precision remains stable and yields accurate approximations suitable for standard numerical applications. As illustrated in Figure~\ref{fig-db2}, the estimated function $S$ (obtained after 12 iterations of the cascade algorithm) is sufficiently bounded away from zero, ensuring its strict positivity.

\bibliographystyle{abbrv}
\bibliography{biblio.bib}

\end{document}